\documentclass[12pt, a4paper]{amsart}
\usepackage{amsmath,amssymb,amscd,amsfonts}

\usepackage{ifthen}
\usepackage[T1]{fontenc}
\usepackage[utf8]{inputenc}
\usepackage[all]{xy}
\usepackage{graphicx}
\usepackage{enumerate}
\usepackage{xspace}
\usepackage{pdfsync}
\usepackage{epic}
\usepackage{dsfont}

\newtheorem{theorem}{Theorem}[section]
\newtheorem{remark}[theorem]{Remark}
\newtheorem{definition}[theorem]{Definition}
\newtheorem{lemma}[theorem]{Lemma}

\newtheorem{proposition}[theorem]{Proposition}
\newtheorem{conjecture}[theorem]{Conjecture}
\newtheorem{corollary}[theorem]{Corollary}
\newtheorem{example}[theorem]{Example}
\newtheorem{question}[theorem]{Question}

\newcommand\C{{\mathbb{C}}}

\def\cO{{\mathcal O}}

\def\cE{{\mathcal E}}

\def\cG{{\mathcal G}}
\def\cH{{\mathcal H}}
\def\cF{{\mathcal F}}

\def\cQ{{\mathcal Q}}

\begin{document}
\title[]{Algebraicity of foliations on complex projective manifolds, applications}

\author{Fr\'ed\'eric Campana}
\address{Universit\'e Lorraine \\
Nancy \\ }

\email{frederic.campana@univ-lorraine.fr}



\date{\today}

\maketitle

\tableofcontents


\begin{abstract} This is an expository text, originally intended for the ANR `Hodgefun' workshop, twice reported, organised at Florence, villa Finaly, by B. Klingler. 

We show that holomorphic foliations on complex projective manifolds have algebraic leaves under a certain positivity property: the `non pseudoeffectivity' of their duals. This permits to construct certain rational fibrations with fibres either rationally connected, or with trivial canonical bundle, of central importance in birational geometry. A considerable extension of the range of applicability is due to the fact that this positivity is preserved by the tensor powers of the tangent bundle. The results presented here are extracted from \cite{CPa}, which is inspired by the former results (\cite{Mi}, \cite{BM}, \cite{CPe}). In order to make things as simple as possible, we present here only the projective versions of these results, although most of them can be easily extended to the logarithmic or `orbifold' context.\end{abstract}


\section{Algebraicity criterion: statement}

In all the text, $X$ will denote an $n$-dimensional connected complex projective manifold, $A$ will be an ample line bundle on $X$. A Zariski open subset of $X$ is said to be `big' if its complement has codimension at least $2$ in $X$.

Let $\cF\subset TX$, the holomorphic tangent bundle of $X$, be a coherent saturated subsheaf of rank $r\leq n$. 

Recall that $\cF$ is saturated if $TX/\cF$ has no torsion. Since $\cF$ is saturated, the Zariski closed subset $Sing(\cF)\subset X$ over which $\cF$ is not a subbundle has codimension at least two, and any local section of $\cF$ defined on the complement of $Sing(\cF)$ extends across $Sing(\cF)$.

\begin{definition} We say that $\cF$ is a foliation if, for any two germs of sections $V,W$ of $\cF$, their Lie Bracket $[V,W]$ is also a germ of section of $\cF$. Equivalently, the Lie bracket defines a morphism of sheaves of $\cO_X$-modules\footnote{Indeed: $[fV,gW]=fg.[V,W]+(fV(g).W-g.W(f).V)\equiv fg.[V,W]$ modulo $\cF$ for any local sections $V,W$ of $\cF$, and holomorphic functions $f,g$.}$L:\wedge^2 \cF\to TX/\cF$ which vanishes identically on $X$.
\end{definition}

This is the `Frobenius' integrability condition. It implies that, over $U_{\cF}:=X\setminus Sing(\cF)$, each point has a local analytic open neighbourhood $U=F\times B$, with $F,B$ open subsets of $\C^r$ and $\C^{n-r}$ respectively such that $\cF_U=Ker(d\beta)$, where $\beta: U\to B$ is the projection onto the second factor $B$. In particular, for each such $x=(f,b)\in U$, there is a unique germ of manifold $F_x=F\times \{b\}$ of dimension $r$, called the germ of the `leaf' of $\cF$ at $x$, which has in each of its points $x':=(f',b)$, the subspace $\cF_{x'}$ as its tangent space. 

We thus define an equivalence relation on $U_{\cF}$ for which two points are equivalent if they can be connected by a chain of germs of leaves of $\cF$ locally defined as above. The classes of this relation are called the leaves of $\cF$. They are  the maximal connected, $r$-dimensional manifolds immersed (but not necessarily closed) in $U_{\cF}$, and tangent to $\cF$. Near the points of $Sing(\cF)$, the leaves of $\cF$ may have a chaotic behaviour.

\begin{example}\label{ex1} The simplest example of an everywhere regular foliation with non-closed leaves is the following: $X:=\C^2/\Lambda$ is an Abelian surface, and $\cF=X\times \{\C.v\}\subset TX=X\times \C^2$ is the trivial rank one subbundle generated by a vector $v\in \C^2$. Let $\Gamma:=\C.v\cap \Lambda$: this is a closed additive subgroup of $\C.v$ of rank $\rho$ either $0,1$ or $2$. So all leaves are isomorphic to $\C.v/\Gamma$, which is isomorphic to either $\C, \C^*$, or an elliptic curve when $\rho=2$, i.e when $\Gamma$ is a cocompact sublattice of $\C.v$. If the leaves of $\cF$ are not compact, they are Zariski-dense, but their topological closures are real tori of real dimension $d$ either $3$ or $4$. It may happen that $d=3$ even if the leaves of $\cF$ are isomorphic to $\Bbb C$. This happens for example if $v=e_1+\sqrt{2}.e_2$, $J.v=e_1+\sqrt{3}.e_3:=w$, the lattice defining $X$ being generated over $\Bbb Z$ by $e_1,\dots,e_4$, and the complex structure $J$ on $\Bbb R^4$ extending the one defined above ($J.v=w)$. One obtains so a two-dimensional compact complex torus $X$. It is less obvious whether this can be done in such a way that $X$ is an abelian variety.\end{example}

\begin{example}\label{ex1'} Let $X:=\Bbb P^2$, with affine coordinates $(x,y)$ and $\cF$ be defined as the Kernel of the (rational) $1$-form $w:=dy-\lambda.dx$, where $\lambda\in \C^*$. The leaves are then defined by the equation $y=x^{\lambda}$, and are thus algebraic if and only if $\lambda\in \Bbb Q^*$. The Zariski and topological closures of the leaves can be easily described in terms of $\lambda$.  Examples \ref{ex1} and \ref{ex1'} show that algebraicity is an arithmetic property in families of foliations. \end{example}

\begin{example}\label{ex2} The globally simplest foliations, however, are the ones arising from dominant rational fibrations $f:X\dasharrow Z$, where $Z$ is a manifold of dimension $p:=n-r$, and the generic fibre of $f$ is connected. The associated foliation is then $\cF:=(Ker(df))^{sat}$, where $\bullet^{sat}$ denotes the saturation inside $TX$, and $Sing(\cF)$ is the union of the indeterminacy locus of $f$, and of the singularities of the (reduction of the) fibres.
\end{example}

\begin{definition} Let $\cF\subset TX$ be a foliation on $X$. We say that $\cF$ is `algebraic' if all leaves of $\cF$ are algebraic submanifolds of $X$, that is: if their Zariski and topological closures coincide, or equivalently, have the same dimension $r$.

Using Chow-scheme theory, one shows that $\cF$ is algebraic if and only if it arises from a fibration, as in Example \ref{ex2}. 
\end{definition}

\smallskip

{\bf Problem:} Find `numerical' criteria for the algebraicity of foliations $\cF\subset TX$. The criteria we shall give and use below are the positivity of intersection numbers of $det(\cF)$ and `movable' classes of curves on either $X$ or the projectivisation of $\cF$. We shall see that these conditions also imply restrictions on the structure of the leaves of $\cF$.

\medskip

The criterion we shall prove and use here is the following, inspired by, and partially extending, former results (\cite{Mi}, \cite{BM}, \cite{CPe}):

\begin{theorem}\label{algcrit} (\cite{CPa}, \cite{Dr}) The foliation $\cF\subset TX$ is algebraic if the dual $\cF^*:=Hom_{\cO_X}(\cF,\cO_X)$ of $\cF$ is not pseudo-effective\footnote{This is proved, but not stated explicitely in \cite{CPa}. It is stated explicitely in \cite{Dr}.}.\end{theorem}

We also have the next particular case:

 \begin{theorem}\label{Mu-alg} Let $\cF\subset TX$ be a foliation with $\mu_{\alpha,min}(\cF)>0$ for some $\alpha\in Mov(X)$. 
Then $\cF^*$ is not pseudo-effective, $\cF$ is an algebraic foliation. Moreover, its leaves have rationally connected closures.
  \end{theorem}
  
  The notions of rational connectedness and $\mu_{\alpha,min}$ will be defined in sections 4 and 6 below. We now define the pseudo-effectivity.

\begin{definition} A coherent sheaf $\cG$ on $X$, is pseudo-effective if, for any $j>0, c>0$, $H^0(X,Sym^{[m]}(\cG)\otimes A^j)\neq\{0\}$ for some $m>j.c$. Here $Sym^{[m]}(\cG)$ denotes the reflexive hull (i.e: the double dual) of $Sym^m(\cG)$. 

Otherwise, if $H^0(X, Sym^{[m]}(\cG)\otimes A^j)=\{0\}, \forall m>c(A).j$, any $j>0$, and some constant $c(A)>0$, $\cG$ is said to be not pseudo-effective.\end{definition}

Equivalently, by the next remark, $\cG$ is not pseudo-effective if $\cG$ is locally free on $U=X\setminus S$, where $S\subset X$ is Zariski closed of codimension at least $2$,  and if $H^0(U,Sym^{m}(\cG)\otimes A^j)=\{0\}, \forall m>c(A).j$.

\begin{remark} Recall that $\cF$ is reflexive if the natural morphism $\cF\to \cF^{**}$ is an equality. Such a sheaf is torsionfree and normal (i.e: any section defined on a big Zariski open subset $U\subset X$ extends to $X)$. If $rk(\cF)=1$, reflexive means locally free on a smooth $X$. If $\cF\subset \cE$ is of rank one, with $\cE$ locally free, then $\cF$ is reflexive if and only if it is saturated in $\cE$ (but not necessarily a subbundle). Its powers $\cF^k\subset Sym^k\cE,k>0$ are then all saturated.
\end{remark}

We shall see in \S\ref{mov} and \S\ref{sslope} how to check the non-pseudo-effectivity of a sheaf $\cG$ using negativity or non-positivity of intersection numbers with `movable classes of curves'.

The following consequences of Theorem \ref{algcrit} can be stated without reference to pseudo-effectivity.

\begin{corollary}\label{ur}(\cite{CPa},\cite{CPe}) Let $s$ be a nonzero section of $\otimes^mTX\otimes L$, for some $m>0$ and $L$ a line bundle with $c_1(L)=0$. Assume that $s$ vanishes somewhere on $X$. Then $X$ is uniruled (i.e: covered by rational curves).
\end{corollary}

Recall that a line bundle on $X$ is said to be `big' if $h^0(X,m.L)\sim C.m^n$, for some $C=C(X,L)>0$, and $m\to +\infty$. A line bundle $L$ is big if and only if $mL=A+E$ for some $m>0, A$ ample and $E$ effective.

\begin{corollary}\label{cbig} (\cite{CPa}) Let $L$ be big line bundle on $X$, and assume there exists a sheaf injection $L\subset \otimes^m \Omega^1_X,$ for some $m>0$. Then $K_X$ is big.\end{corollary}

\begin{remark} The statement of Corollary \ref{cbig} is a form of stability specific to vector bundles $E:=\otimes^m\Omega^1_X$, and fails for general $E$: on any positive-dimensional $X$, if $E:=A\oplus A^{\otimes -2}$, with $L=A$ ample of rank one, $det(E)=-A$, although $A$ injects in $E$. 

Products $X=\Bbb P^1\times Z_{n-1}$, with $Z_{n-1}$ of general type, $f:X\to Z$ the projection, and $L=f^*(K_Z)\subset \Omega^{n-1}_X$, with  $\kappa(X,L)=(n-1)$ submaximal, show that $K_X$ may not be pseudo-effective. The positivity of subsheaves $L\subset \otimes^m  \Omega^1_X$ is preserved by $K_X$ only when $L$ is big. 

We shall formulate a more general version of `birational stability' for tensor powers of cotangent bundles in \S.\ref{birst}.
\end{remark} 

We shall apply in \S\ref{SVconj} the Corollary \ref{cbig} to moduli of canonically polarized manifolds.


\section{Proof of the algebraicity criterion.}

Let $S:=Sing{\cF}\subset X$ be the Zariski closed set of codimension at least $2$ over which $\cF$ is not a subbundle of $TX$, and $U_{\cF}:=X\setminus S$. For any $x\in U_{\cF}$, let $U_x=F_x\times B_x$ be an open neighbourhood $U_x$ of $x$ in $X$ such that the leaves of $\cF_{\vert U_x}$ are the fibres of the projection $\beta_x:U_x\to B_x$. If $\Delta\subset X\times X$ is the diagonal of $X$, and if $\Delta_{\cF}$ is the Zariski open set of $\Delta$ mapped to $U_{\cF}$ by (either) projections $p_i:X\times X\to X, i=1,2$, then there exists a germ $W$ of submanifold of dimension $n+r$ of $U_{\cF}\times U_{\cF}$ such that for each $x\in U_{\cF}$, $W\cap U_x\times U_x=U_x\times_{B_x}U_x$. In other words, the fibre $W_x$ over $x\in U_{\cF}$ of the first projection $p_{1\vert W}:W\to U_{\cF}$ is a germ of the leaf of $\cF$ through $x$.

$\bullet$ The leaves of $\cF$ will thus be algebraic if (and only if) the manifold $W$ is a germ of an algebraic subvariety of $X\times X$, or equivalently, if the Zariski closure $V$ of $W$ has the same dimension $n+r$ as $W$. Let $d:=dim(V)$; if $B$ is any ample line bundle on $V$, there exists a constant $C=C(B,V)>0$ such that $h^0(V, k.B)\cong C. k^d$, when $k\to +\infty$.
We thus only need to prove that $h^0(V,k.B)\leq C.k^{n+r}$ for some $C>0$, and some ample $B$ on $V$ in order to prove Theorem \ref{algcrit}.

$\bullet$ Notice next that it is sufficient to show that $h^0(W,k.B)\leq C.k^{n+r}$ for some $C>0$, since the restriction map $res: H^0(V,k.B)\to H^0(W,k.B)$ is injective by the Zariski density of $W$ in $V$. We now prove this last inequality.

$\bullet$ Let thus $B$ be ample on $X\times X$. Any section $s$ of $k.B$ on $W$ admits a unique development in power series along the fibres of $p_{1\vert W}:W\to U_{\cF}$ of the form $s=\sum_{m\geq 0} s_{m,k}$, where $s_{m,k}\in H^0(U_{\cF}, Sym^m(\cF^*)\otimes k.B), \forall m,k$.

Since $\cF^*$ is not pseudo-effective, $H^0(U_{\cF}, Sym^m(\cF^*)\otimes k.B)=\{0\}$ if $m>C'.k$, for some $C'=C'(B)$. 

Thus $h^0(W,k.B)\leq \sum_{m=0}^{m= C'.k}h^0(U_{\cF},Sym^m(\cF^*)\otimes kB)$.

If $U_{\cF}=X$ (i.e: $\cF$ is a subbundle of $TX)$, $h^0(U_{\cF}, Sym^m(\cF^*)\otimes k.B)\leq C". (m+k)^{n+r-1}, \forall k>0, \forall  m\geq 0$, and some $C">0$. 
Indeed, $h(m,k):=h^0(X, Sym^m(\cF^*)\otimes k.B)=h^0(\Bbb P(\cF), m.L+k.B'))$, where $L=\mathcal{O}_{\Bbb P(\cF)}(1)$, $B'$ is the pullback of $B$, and so $h(m,k)\leq C".(m+k)^{n+r-1}$, since $dim(\Bbb P(\cF))=n+r-1$, by dominating $L,B'$ by an ample line bundle $A'$ on $\Bbb P(\cF)$ such that $A'-B'$ and $A'-L$ are effective. Thus: 

$\sum_{m=0}^{m= C'.k}h^0(U_{\cF},Sym^m(\cF^*)\otimes kB)\leq \sum_{m=0}^{m=C'.k} C".(m+k)^{n+r-1}\leq (C'k+1).(C".(C'k+k)^{n+r-1})$$\leq (C'+1)^{n+r}C".k^{n+r}$, which concludes the proof of Theorem \ref{algcrit} if $U=X$.

The general case $U_{\cF}\neq X$ is obtained by applying \cite{Nak}, III.5.10(3), which shows that the same inequalities $h^0(U_{\cF}, Sym^m(\cF^*)\otimes k.B)\leq C". (m+k)^{n+r-1}, \forall k>0, \forall  m\geq 0$, and some $C">0$ still hold true, by constructing a suitable modification of $\Bbb P(\cF)$, and of $\mathcal{O}_{\Bbb P(\cF)}(1)$.

More precisely: if $p:P:=\Bbb P(\cF):=Proj(Sym^{\bullet}\cF)\to X$, with $L:=\mathcal{O}_{P}(1)$ the tautological line bundle on $P$, $p_*(L^m)=Sym^m(\cF),\forall m>0$. Moreover, if $\rho: P'\to P$ is a resolution of the singularities of $P$ which coincides with $P$ over $U_{\cF}$, there exists an effective, $\rho$-exceptional divisor $E\subset P'$ such that $(p\circ r)_*(m.(L+E))=Sym^{[m]}(\cF),\forall k>0$. See Proposition \ref{reflex} for a simplified proof.


 \section{Pseudoeffectivity and movable classes.}\label{mov}

 \begin{definition} Let $(C_t)_{t\in T}$ be an algebraic family of curves parametrised by an irreducible projective variety $T$. Assume that $C_t$ is irreducible for $t\in T$ generic, and that $X$ is covered by the union of the $C_t's$. Then $\alpha:=[C_t]\in H^{2n-2}(X,\Bbb Z)$ is independent of $t\in T$. We call such a class a {\bf geometrically movable class} on $X$.
 
 The closed convex cone of $H^{2n-2}(X,\Bbb R)$ generated by the geometrically movable classes is called the {\bf movable cone} of $X$, denoted $Mov(X)$, and its elements are the {\bf movable classes} of $X$. From \cite{BDPP}, we even have that $\alpha\in H^{n-1,n-1}(X,\Bbb R)$ is in $Mov(X)$ if and only if $\alpha.D\geq 0$ for any irreducible effective divisor $D$ on $X$.
 \end{definition}

 \begin{example}\label{exmov} 1. If $A$ is ample on $X$, then $A^{n-1}\in Mov(X)$. There are many more examples, such as the following:
 
 2. If $C\subset X$ is an irreducible curve, locally complete intersection, with ideal sheaf  $I_C$ and ample normal bundle $(I_C/I_C^2)^*$, then $\alpha:=[C]\in Mov(X)$. Indeed: for any irreducible effective divisor $D$, we have $D.C\geq 0$, since $D_{\vert C}$ is a quotient of $(I_C/I_C^2)^*$. See \cite{O}, \cite{P} for more on this situation.
 \end{example}

 \begin{remark}If $D$ is an effective $\Bbb Q$-divisor on $X$, $D.\alpha\geq 0,\forall \alpha \in Mov(X)$. More generally, if $m.D+A$ if effective for infinitely many $m>0$, and some ample $A$,  $D.\alpha\geq 0,\forall \alpha\in Mov(X)$. Said otherwise: if $D$ is pseudo-effective, then $D.\alpha\geq 0,\forall \alpha\in Mov(X)$.
  \end{remark}
 
 An important result is the following converse.

 \begin{theorem}\label{pseff}(\cite{BDPP},\cite{Nak}) Let $L$ be a line bundle on $X$. Then $L$ is pseudo-effective if and only if  $L.\alpha\geq 0,\forall \alpha \in Mov(X)$. 
 
 Equivalently, $L$ is {\bf not} pseudo-effective if and only if $L.\alpha<0$ for some $\alpha \in Mov(X)$.
 \end{theorem}
 
 Pseudo-effectivity is fundamental in birational geometry when applied to $L=K_X$ because of the next:
 
 \begin{theorem}(\cite{MM}) If $K_X$ is not pseudo-effective, then $X$ is uniruled (i.e: covered by rational curves). \end{theorem}
 
 The only known proof rests on positive characteristic techniques. The original statement of Miyaoka-Mori is: if $(C_t)_{t\in T}$ is an algebraic family of curves (of possibly large genus) such that $K_X.C_t<0$, it is possible to find such a family where the $C_t's$ are rational curves!
 
\begin{remark} There exists (Nagata, Mumford, see \cite{mum}) divisors $D$ such that $D.\alpha\geq 0,\forall \alpha\in Mov(X)$, but which are not $\Bbb Q$-effective. 
The main conjecture of birational geometry is that $K_X$ is $\Bbb Q$-effective if $K_X.\alpha\geq 0,\forall \alpha\in Mov(X)$.
 \end{remark} 
 
 The following birational invariance is crucial, here.
 
 \begin{corollary}\label{bir} Let $\pi:X'\to X$ be a birational morphism between smooth projective manifolds. 
 
 1. Let $\alpha \in Mov(X)$, then $\pi^*(\alpha)\in Mov(X')$. 
 
 2. Let $\alpha' \in Mov(X')$, then $\pi_*(\alpha')\in Mov(X)$. 
 
 3. $\alpha\in Mov(X)$ if and only if $\alpha':=\pi^*(\alpha)\in Mov(X')$.
 \end{corollary} 
 
 \begin{proof} The first (resp. second) claim follows from the equality: $\pi^*(\alpha).D'=\alpha.\pi_*(D')$ for any divisor $D'$ on $X'$ (resp. $\pi^*(\alpha).\pi^*(D)=\pi^*(\alpha).(D'+E')=\pi^*(\alpha).D'$, where $D'$ is the strict transform of $D$, and $E'$ is $\pi$-exceptional, since $\pi^*(\alpha).E'=0,\forall E'$, $\pi$-exceptional divisor on $X'$. The third claim follows from the others since $\pi_*(\pi^*(\alpha))=\alpha,\forall \alpha$. 
 \end{proof}
 
 The following is an immediate consequence of Theorem \ref{pseff} since $p^*(A)+\varepsilon.\mathcal{O}_{\Bbb P(\cE)}(1)$ is ample on $\Bbb P(\cE)$ for any $\varepsilon>0$ sufficiently small, and $Sym^m(\cE)=p_*(\mathcal{O}_{\Bbb P(\cE)}(m)),\forall m>0$:

 \begin{corollary}\label{Epseff} Let $\cE$ be a rank $r$ vector bundle on $X$. Then $\cE$ (identified with its sheaf of sections) is pseudo-effective if and only if $L:=\mathcal{O}_{\Bbb P(\cE)}(1)$ is pseudo-effective on $p:\Bbb P(\cE)\to X$, that is, if $L.\alpha\geq 0, \forall \alpha\in Mov(\Bbb P(\cE))$.
 \end{corollary}
 
The above Corollary \ref{Epseff} extends to the case when $\cG=E$ is not locally free by considering a suitable modification of $\Bbb P(\cG)$ as constructed in \cite[III,5.10]{Nak}. We give a simplified proof, following the strategy of \cite{Nak}, in Proposition \ref{reflex} below (to be applied to a smooth model $P$ of the main component of $\Bbb P(\cF)$, as in \cite[V,3.23]{Nak}).

\begin{proposition}\label{reflex} Let $p:P\to X$ be a fibration with $X$ normal, $P$ smooth, both projective, let $D$ be a Cartier divisor on $P$ such that $p_*(\cO_P(D))\neq 0$. 

1. If $E$ is an effective divisor on $P$ supported on the exceptional divisor $Exc(p)$ of $p$, and such that, for every divisorial component $\Gamma$ of $Exc(p)$, its restriction $E_{\Gamma}$  to $\Gamma$ is not pseudo-effective, then $p_*(\mathcal{O}_P(k.(D+\ell.E)))=p_*(\mathcal{O}_P(k.D))^{**},\forall k>0$, where $\ell>0$ is chosen to be sufficiently large, so that $(D+\ell.E)_{\Gamma}$ is not pseudo-effective for any $\Gamma$.

2. There exists a divisor $E$ satisfying the previous properties. 
\end{proposition}

\begin{proof} 1. 
Let $\Sigma\subset X$ be the codimension $2$ or more locus of points over which $p$ is not equidimensional, so that $Exc(p)\subset p^{-1}(\Sigma)$. There is some $\ell>0$ such that $(D+\ell.E)_{\Gamma}$ is not pseudo-effective for every $\Gamma$, since this is the case for $E$. We replace $E$ by $\ell.E$, so that $\ell=1$, to simplify notations. Then every section of $\cO_P(k.(D+E))$ vanishes on $Exc(p)$, for every $k>0$, and has thus no pole there. Every element $s$ of $H^0(P,p^*(p_*(\cO_P(k.D))^{**}))$ lies in $H^0(P,k.(D+E)+E')$, for some divisor $E'$ supported on $Exc(p)$, which may be supposed to be effective. But then $E'\leq t.E$ for some $t>0$ since the support of $E$ coincides with the divisorial part of  $Exc(p)$ by our non pseudo-effectivity assumption. Since $t.E$ is not pseudo-effective on each $\Gamma$, $H^0(P,k.(D+E)+t.E)=H^0(P,k.(D+E))=H^0(P, p^*(p_*(\cO_P(k.D))^{**}))$, which implies the claim.

2. Let $v:X'\to X$ be a modification, $q:P'\to X', u:P'\to P$ be the normalisation of the main component of $P\times_X X'$, so that $p\circ u=v\circ q$. We choose $v$ to be projective birational, with $X'$ smooth, and moreover  such that $q$ is equidimensional, and such that there exists an effective divisor $\Delta\subset Exc(v)$ such that $-\Delta$ is ample on the divisorial part of $Exc(v)$. 

Let $E:=u_*(q^*(\Delta))$, and let $H$ be a sufficiently ample divisor on $P$, with $H_{\Gamma}$ its restriction to $\Gamma$, for each divisorial component $\Gamma$ of $Exc(p)$. Let $\alpha_{\Gamma}:=H_{\Gamma}^{d-2}$, where $d:=dim(P)$, then $\alpha_{\Gamma}\in Mov(\Gamma)$. If $\Gamma'\subset P'$ is the strict transform of $\Gamma$ by $u$, then $q^*(\Delta).u^*(\alpha_{\Gamma})=E.\alpha_{\Gamma}$ since the generic member of the family $\alpha_{\Gamma}$ does not meet the indeterminacy locus of $u_{\vert \Gamma'}:\Gamma'\to \Gamma$. We thus just have to show that $q^*(\Delta).u^*(\alpha_{\Gamma})<0$.

Moreover, $q^*(\Delta).u^*(\alpha_{\Gamma})=\Delta.q_*(u^*(\alpha_{\Gamma}))<0$. Indeed: $q_*(u^*(\alpha_{\Gamma}))\in Mov(q(\Gamma'))$, $-\Delta$ is ample on the divisorial part of $Exc(v)$, and $q(\Gamma')$ is a divisorial component of $Exc(v)$, by the equidimensionality of $q$. Thus $E.\alpha_{\Gamma}<0$, and $E_{\Gamma}$ is not pseudo-effective.

\end{proof}

 The description of $Mov(\Bbb P(E))$ is, in general, quite delicate. An important particular situation where the pseudo-effectivity can be tested on $X$ rather than on $\Bbb P(E)$ is exposed in the next section.


 \section{Harder-Narasimhan filtrations and pseudo-effectivity.}\label{sslope}
 
 The data here are $X,\cG,\alpha\in Mov(X),r:=rk(\cG)>0$ as above. We assume always $\cG$ to be nonzero, torsionfree. See \cite{Kob}, Chap. V and \cite{GKP'} for a detailed treatment.
 
 The (always locally free, since $X$ is smooth), rank-one, sheaf $det(\cG)$ is defined as $det(\cG):=\wedge^r(\cG)^{**}$.
 
 \begin{definition} The slope of $\cG$ with respect to $\alpha$ is: $\mu_{\alpha}(\cG):=\frac{det(\cG).\alpha}{r}$. We say that $\cG$ is $\alpha$-stable (resp. $\alpha$-semi-stable) if: $\mu_{\alpha}(\cH)<\mu_{\alpha}(\cG)$ (resp.$\mu_{\alpha}(\cH)\leq \mu_{\alpha}(\cG))$,$\forall \cH\subsetneq \cG$ coherent subsheaf. \end{definition}

 Notice that if $\cF\subsetneq \cG$ is a saturated subsheaf of rank $r$ with torsionfree quotient $\cQ$ of rank $s$, we have: $det(\cG)=det(\cF)+det(\cQ)$ and so: 
 
 $(*)$ $\mu_{\alpha}(\cG)=\frac{r}{r+s}.\mu_{\alpha}(\cF)+\frac{s}{r+s}.\mu_{\alpha}(\cQ)\in [\mu_{\alpha}(\cF),\mu_{\alpha}(\cQ)]$.
 
 If $\cF$ is torsionfree, then $\mu_{\alpha}(\cF)=\mu_{\alpha}(\cF^{**})=-\mu_{\alpha}(\cF^*)$, since these equalities need to be checked only on a big Zariski-open subset of $X$.

 \medskip

 These notions are classical when $\alpha=[A]^{n-1}$, where $A$ is an ample line bundle. Just as in this classical case, we still have (with essentially the same proof, by induction on the rank, and the equality $(*)$):
 
 \begin{lemma} For $\cG,\alpha$ as above, there is a unique maximum $\cH\subset \cG$ with $\mu_{\alpha}(\cH)$ maximum (i.e: for any $\cH'\subset \cG$, we have: $\mu_{\alpha}(\cH')\leq \mu_{\alpha}(\cH):=\mu_{\alpha,max}(\cG)$, in case of equality: $\cH'\subset \cH).$ 
 
 $\cH$ is the $\alpha$-maximal destabilising subsheaf of $\cG$, denoted $\cG^{\alpha,max}$. It is, by construction, $\alpha$-semi-stable, and saturated inside $\cG$, so its quotient $\cG_{\alpha,min}:=\cG/\cG^{\alpha,max}$ is torsionfree (or zero iff $\cG$ is semi-stable). 
 
 
\end{lemma}
 
 \begin{proof} Obvious if  $rank(\cG)=1$, and so for direct sum of copies of a line bundle. Then embedd $\cG$ in the direct sum of a certain number of copies of $A$, sufficiently ample. Then $\mu_{\alpha}(\cG')\leq \mu_{\alpha}(\oplus^NA)=\mu_{\alpha}(A)$ for any $\cG'\subset \cG$. This proves the boundedness of $\mu_{\alpha}(\cG')$, for $\cG'\subset \cG$. If the class $\alpha$ is rational, one easily gets the existence of a maximum for these slopes, because of the finiteness of the possible denominators, and one can choose the rank to be maximal. The conclusion then follows. If the class $\alpha$ is not rational, one needs a further (but still elementary) argument: choose $\cG_i,i=1,2$ both different from their intersection, of the same maximal rank $r$ with $\alpha$-slopes $\mu_i, i=1,2$, approaching  from below up to $\varepsilon>0$ the upper bound $\mu:=\mu_{\alpha,max}(\cG)$ of $\alpha$-slopes of all $\cG'\subset \cG$. The sum $\cG_1+\cG_2$ has rank larger than $r$ and slope at least each  $\mu-2r.\varepsilon$ (by a simple computation using the exact sequence $0\to \cG_1\cap \cG_2\to \cG_1\oplus \cG_2\to \cG_1+\cG_2\to 0$). Contradiction to the maximality of $r$ for the ranks of $\cG'\subset \cG$ with slope at least $\mu-2r.\varepsilon$ if $\varepsilon>0$ is chosen sufficiently small.  \end{proof}

 \begin{corollary}
 There are a unique integer $s\geq 0$, and an increasing filtration $\{0\}=\cH_0\subsetneq \cH_1\subsetneq \cH_2\subsetneq \dots\subsetneq \cH_s=\cG$ by saturated subsheaves such that: $\cH_{j+1}/\cH_{j}$ is $\alpha$-semistable for $j=0,\dots,s-1$, and $\mu_{\alpha}(\cH_{j+1}/\cH_{j})> \mu_{\alpha}(\cH_{j+2}/\cH_{j+1})$ for $j=0,\dots,s-2$ . 
 
 This filtration is called the $\alpha$-Harder-Narasimhan filtration of $\cG$. We write $\mu_{\alpha,min}(\cG):=\mu_{\alpha}(\cG_{\alpha,min}):=\mu_{\alpha}(\cG/\cH_{s-1})$, and $\mu_{\alpha,max}(\cG):=\mu_{\alpha}(G^{\alpha,max})$.\end{corollary}

 \begin{proof} Induction on $s$ applied to $\cG/\cG^{\alpha,max}$ and $\cH_1=\cG^{\alpha,max}$. \end{proof}

 \begin{lemma}\label{slope} 1. Let $\cH\neq \{0\}$ be a torsionfree quotient sheaf of $\cG$ on $X$. Then: $\mu_{\alpha}(\cH)\geq \mu_{\alpha,min}(\cG)$. If equality holds, $\cH$ is a quotient of $\cG_{min}$.
 
 2. Let $\cF,\cG$ be torsionfree coherent sheaves on $X$. Assume that $\mu_{\alpha,min}(\cF)>\mu_{\alpha,max}(\cG)$. Then $Hom(\cF,\cG)=\{0\}$.
 
 3. If $det(\cG).\alpha>0$, then $\mu_{\alpha,max}(\cG)=\mu_{\alpha,min}(\cG^{\alpha,max})>0$.

 \end{lemma}

 \begin{proof} 1. Induction on $r$: if $\cH$ is a nonzero quotient of $\cG$, it fits in a short exact sequence $0 \to K \to \cH \to \cQ \to 0$, in which $K$ (resp. $\cQ)$ is a quotient of $\cG^{\alpha,max}$ (resp. $\cG/\cG^{\alpha,max})$, hence the conclusion since both $K$ and $\cQ$ have slope at least $\mu_{\alpha,min}(\cG)$. The second claim holds by induction on $s$, because $\mu_{\alpha,min} (\cG/\cG^{\alpha,max})>\mu_{\alpha,min}(\cG)$ if $\cG$ is not semi-stable.

 2. If $0\neq h\in Hom(\cF,\cG)$, its image $\cH$ a subsheaf of $\cG$, and a quotient of $\cF$. We thus get (with $\mu=\mu_{\alpha})$: $\mu_{max}(\cG)\geq \mu(\cH)\geq \mu_{min}(\cF)$, a contradiction. 
 
 3. is obvious.
 \end{proof}

 \begin{lemma} We have: $(\cG^*)^{\alpha,max}=(\cG_{\alpha, min})^*$, in particular: $\mu_{\alpha,min}(\cG)=-\mu_{\alpha,max}(\cG^*)$, and $HN_{\alpha}(\cG^*)=(HN_{\alpha}(\cG))^*$ (i.e: the terms for $\cG^*$ are the dual of those for $\cG$ in reverse order).
 \end{lemma}
 
 \begin{proof} Dualising the projection $\cG\to \cG_{min}$, we get: $(\cG_{\min})^* \subset \cG^*$, and so $-\mu_{min}(\cG)\leq \mu_{max}(\cG^*)$. Dualising $(\cG^*)^{\alpha,max}\subset \cG^*$, we get a generically surjective morphism $\cG \to \cG^{**} \to ((\cG^*)^{\alpha,max})^*$ from which the inequality $-\mu_{max}(\cG^*)\geq \mu_{min}(\cG)$ follows. We thus get the equality $-\mu_{max}(\cG^*)= \mu_{min}(\cG)$. Hence the second claim. The first claim follows from the second part of Lemma \ref{slope}.(1). The last claim is seen by induction on the number of terms of the HN filtration. \end{proof}
 
\medskip

Let $\cF\hat{\otimes}\cG:=(\cF\otimes \cG)^{**}$ be the reflexive tensor product of $\cF$ and $\cG$.
 
 \begin{theorem}\label{mu} For $X,\cG,\cH,\alpha$ as before, we have:
 
 1.  $\mu_{\alpha,min}(\cG)=-\mu_{\alpha,max}(\cG^*)$, $(\cG_{\alpha,min})^*=(\cG^*)^{\alpha,max}$. 
 
 2. $(\cG\hat{\otimes} \cH)^{\alpha,max}=\cG^{\alpha,max}\hat{\otimes} \cH^{\alpha,max}/$.
 
 3. $\mu_{\alpha,max}(\cG\hat{\otimes} \cH)=\mu_{\alpha,max}(\cG)+\mu_{\alpha,max}(\cH)$.
 
 4.$ \mu_{\alpha,max}(\hat{\otimes}^m(\cG))=\mu_{\alpha,max}(\widehat{Sym}^m(\cG))=m. \mu_{\alpha,max}(\cG)$.
 
 5.  $\mu_{\alpha,max}(\hat{\wedge}^p(\cG))=p. \mu_{\alpha,max}(\cG),\forall p>0$.
 \end{theorem}
 
 The first claim follows from the fact that dualisation exchanges subobjects and quotients. By contrast, the proof of claim 3 is quite deep, relying on the Kobayashi-Hitchin correspondance between $\alpha$-stable bundles and Hermite-Einstein vector bundles (with respect to Gauduchon metrics). For this correspondance, originally due to Li-Yau in \cite{LY}, see details in  the book of L\"ubke-Teleman \cite{T}.

\begin{lemma}\label{notpseff} Let $\cF$ be a torsionfree coherent sheaf on $X$, and $\alpha\in Mov(X)$ such that $\mu_{\alpha,min}(\cF)>0$. Then $\cF^*$ is not pseudo-effective.
\end{lemma} 

\begin{proof} Let $A$ be ample on $X$.  Applying Theorem \ref{mu}, we see that $-\mu:=\mu_{\alpha,max}(\cF^*)<0$, and so $\mu_{\alpha,max}(Sym^{[m]}(\cF)\otimes A^j)=m. \mu+j.A.\alpha<0$ for $m>j.c$, where $c:=\frac{A.\alpha}{\mu}$, which implies that $h^0(X,Sym^{[m]}(\cF)\otimes A^j)=0$, as claimed.
\end{proof}

The following  criterion for foliations among distributions is crucial here:

 \begin{corollary}\label{miy} (\cite{Mi}) Let $\cF\subset TX$ be a coherent subsheaf. If, for some $\alpha\in Mov(X)$, $2.\mu_{\alpha,min}(\cF)>\mu_{\alpha,max}(TX/\cF)$, then $\cF$ is a foliation. 
 \end{corollary}
 
 \begin{proof} Let $\Lambda: \wedge^2(\cF)\to TX/\cF$ be the sheaf morphism induced by the Lie bracket. The slopes assumption implies by claim 5 of Theorem \ref{mu} and claim 2 of Lemma \ref{slope} that it vanishes. The conclusion follows by Frobenius theorem.
 \end{proof}
 
 From Theorem \ref{algcrit}, and Lemmas \ref{notpseff} and \ref{miy}, we get:

 \begin{theorem}\label{mu-alg} Let $\cF\subset TX$ be a distribution with $\mu_{\alpha,min}(\cF)>0$ for some $\alpha\in Mov(X)$. 
 Assume that $2.\mu_{\alpha,min}(\cF)\>\mu_{\alpha,max}(TX/\cF)$ (this is satisfied if $\cF$ is a piece of the Harder-Narasimhan filtration of $TX$ relative to $\alpha$).
 Then $\cF$ is an algebraic foliation.
  \end{theorem}
 
 \begin{remark} In the situation of Theorem \ref{mu-alg}, we shall prove in the next two sections that the leaves of $\cF$ have rationally connected closures. 
 \end{remark}

 \begin{corollary}\label{mumin>0} Assume that $K_X$ is not pseudo-effective. There then exists on $X$ a nonzero foliation $\cF\subset TX$ such that $\mu_{\alpha,min}(\cF)>0$, the dual $\cF^*$ is not pseudo-effective, and $\cF$ is algebraic. \end{corollary}

 \begin{proof} Just apply Theorem \ref{mu-alg}, choosing $\cF:= TX^{\alpha,max}$. The last claim requires new notions and techniques presented in the next two sections. \end{proof} 
 
 \begin{remark} We thus obtain the existence of (all of the) fibrations with rationally connected fibres on any given $X$ with $K_X$ not pseudo-effective without using the Minimal Model Program. 
  \end{remark}
 Another easy, but central, property is the birational invariance `up' of the condition $\mu_{\alpha,min}(\cF)>0$ for subsheaves of the tangent bundles:
 
 \begin{lemma}\label{birup} Let $\cF\subset TX$ be a subsheaf, let $\pi:X'\to X$ be a birational morphism between $X,X'$ projective smooth. Let $\cF'\subset TX'$ be the pullback distribution, defined as: $\cF':=\pi^*(\cF)\cap TX'$. Let $\alpha\in Mov(X),\alpha':=\pi^*(\alpha)$. Then $\mu_{\alpha'}(\cF')=\mu_{\alpha}(\cF)$, and $\mu_{\alpha',min}(\cF')=\mu_{\alpha,min}(\cF)$. \end{lemma} 
 
 \begin{proof} The statements hold for $\pi^*(\cF)$, hence for $\cF'$ since $det(\cF')$ and $det(\pi^*(\cF))$ coincide outside the exceptional divisor of $\pi$, and $\alpha'.E=0$ for each component $E$ of this divisor. \end{proof}
 
 \begin{remark} The statement does not (always) hold `down' (i.e: if $\alpha=\pi_*(\alpha'),$ but $\alpha'\neq \pi^*(\alpha)$:  let $X=\Bbb P^2$, let $\cF$ be tangent to the conics $C$ going through $4$ points in general position, and let $\alpha'$ be the class of the strict transforms $C'$ of these conics on the blow-up $X'$ of $\Bbb P^2$ in the $4$ given points.
Then $\mu_{\alpha'}(\cF')=2>0$, since $\cF'_{\vert C'}=-K_{C'}$, but $\mu_{\alpha}(\cF)=2-4=-2>0$. \end{remark}


\section{Pseudo-effectivity of relative canonical bundles.} 

Let $f:X\to Z$ be a fibration between complex projective manifolds. Let $X_z$ be its generic (smooth) fibre. Let $f^*(\Omega^1_Z)^{sat}\subset \Omega^1_X$ be the saturation of $f^*(\Omega^1_Z)$, and denote by $f^*(K_Z)^+:=det(f^*(K_Z)^{sat})$. Let finally $K^-_{X/Z}:=K_X-f^*(K_Z)^+$.

If $\cF:=Ker((df)^{sat})\subset TX$, we thus have: $det(\cF^*)=K^-_{X/Z}$, by the consideration of the determinants in the exact sequence: $0\to (f^*(\Omega_Z^1))^{sat}\to \Omega^1_X\to \cF^*\to 0$, dual to: $0\to \cF\to TX\to TX/\cF\to 0$.

\begin{definition}\label{neat} We say that $f$ is `neat' if:

1. The discriminant locus $D(f)\subset Z$ of singular fibres of $f$ is a divisor of strict normal crossings.

2. $f^{-1}(D(f))\subset X$ is also a divisor of simple normal crossings.

3. There exists an equidimensional fibration $f_0:X_0\to Z$ and a birational morphism $u:X\to X_0$ such that $f_0\circ u= f$.
\end{definition}

 By suitable blow-ups of $Z$ and $X$, any fibration can be made `neat', as a consequence of Hironaka's  resolution of singularities, and Hironaka's or Raynaud's flattening theorems.

\begin{theorem}\label{KX/Zpseff} Let $f:X\to Z$ be a fibration between complex projective manifolds. Let $X_z$ be its generic (smooth) fibre. If $f$ is `neat', and if  $K_{X_z}$ is pseudo-effective,  $K^-_{X/Z}$ is pseudo-effective.
\end{theorem} 

This result strengthens a weak version of Viehweg's theorem on the weak positivity of direct images of pluricanonical  sheaves. A particular case was also obtained by Miyaoka using deformations of rational curves, Mori's bend and break, and positive characteristic methods.

\begin{corollary}\label{KXz<0} Assume that $K_X$ is not pseudo-effective, let $\alpha\in Mov(X)$  be such that $K_X.\alpha<0$. Let $\cF:=TX^{\alpha,max}\subset TX$, thus $\mu_{\alpha,min}(\cF)>0$, and so $\cF$ is a foliation, by corollary \ref{mumin>0}, and algebraic, by Theorem \ref{algcrit}. Moreover, after suitable blow-ups of $X$, we have: if $f:X\to Z$ is a fibration such that $\cF=Ker(df)^{sat}\subset TX$, then:

1.  $K_{X_z}$ is not pseudo-effective. More generally:

2. If $f=h\circ g$ for rational fibrations $g:X\to Y$ and $h:Y\to Z$, then $K_{Y_z}$ is not pseudo-effective, $Y_z$ being the generic fibre of $h$ (on suitable birational neat models of $g,h)$.
\end{corollary} 

\begin{proof} We may blow-up $X,Z_0$ for any rational fibration $f_0:X\dasharrow Z_0$ with generic fibre the closure of a generic leaf of $\cF$ in such a way that the birational model $f:X\to Z$ (we keep the notation $X$ for the blown-up manifold, so as to save notations) so obtained is regular and `neat'. By Lemma \ref{birup}, the property $\mu_{\alpha,min}(\cF)$ is preserved on the blown-up manifold by lifting both the foliation and the movable class in the natural manner, so that the Theorem \ref{KX/Zpseff} can be applied to the new $f:X\to Z$.

Claim 1. By contradiction, assume not. Then $K^-_{X/Z}=det(\cF^*)$ is pseudo-effective, by Theorem \ref{KX/Zpseff}. Thus $0\leq K^-_{X/Z}.\alpha=det(\cF^*).\alpha= -\mu_{\alpha}(\cF)<0$. Contradiction. 

Claim 2. By contradiction again, assume  that $K_{Y_z}$ is pseudo-effective. By applying Theorem \ref{KX/Zpseff} to $h$, $det(\cH^*)$, and so also  $g^*(det(\cH)^*)$ were pseudo-effective, with $\cH:=Ker(dh)^{sat}\subset TY$. Thus $\mu_{\alpha}(g^*(\cH^*))=-\mu_{\alpha}(g^*(\cH))\geq 0$. On the other hand, the fibration $g_{\vert X_z}:=g_z:X_z\to Y_z$ implies that $g^*(\cH)$ is a quotient of $\cF$, so $\mu_{\alpha}(g^*(\cH))\geq \mu_{\alpha,min}(\cF)>0$. Contradiction.
\end{proof}

\begin{remark} With (much) more work, it is possible to show more: for general $z\in Z$, there exists $\alpha_z\in Mov(X_z)$ such that $\mu_{\alpha_z,min}(TX_z)>0$. We shall prove this less directly, but more simply, by using the theory of rational curves in the next section.
\end{remark}


\section{Rational curves and non-pseudoeffectivity of the canonical/cotangent  bundles.}

\begin{definition}\label{ur,rc} $X$ is said to be uniruled (resp. rationally connected) if through any generic point (resp. any two generic points) of $X$ pass an irreducible rational curve (i.e: a non-constant image of $\Bbb P^1)$.
\end{definition}

\begin{example} 1. $\Bbb P^n$ is rationally connected, so also any $X$ birational to $\Bbb P^n$, or rationally dominated by $\Bbb P^n$.

2. $X$ is rationally connected if $-K_X$ is ample (i.e: if $X$ is Fano).

3. Smooth hypersurfaces of $\Bbb P^{n+1}$ of degree at most $n+1$ are Fano (by adjunction), hence rationally connected.

4. If $X\to Z$ is a fibration with $Z$ and $X_z$ rationally connected, then $X$ is rationally connected (a deep result of \cite{GHS}).
\end{example}

The present birational theory of rational curves on projective manifolds is based on \cite{GHS} and the following two results:

\begin{theorem}\label{tur}(\cite{MM}) $X$ is uniruled if (and only if) $K_X$ is not pseudo-effective.
\end{theorem}

\begin{theorem}\label{trc} There is a unique fibration $r:X\to R$ (called the Maximally Rationally Connected (or MRC-) fibration of $X)$ such that:

1. its fibres are rationally connected.

2. its base is not uniruled (i.e: $K_R$ is pseudo-effective).
\end{theorem}

Although in general, $r:X\to R$ is rational (almost holomorphic, in the sense that its generic fibre does not meet the indeterminacy locus), since all the properties considered here are of birational nature, we may and shall assume that $r$ is regular. 

The two extreme cases are: $X=R$, or $R$ is a point, meaning respectively that $X$ is not uniruled, and that $X$ is rationally connected.

\medskip

Using Chow-scheme theory, one gets a relative version as well, on suitable birational models:

\begin{corollary} \label{MRCrel}If $f:X\to Z$ is a fibration, there exists $g:X\to Y$ and $h:Y\to Z$ such that for general $z\in Z$, $g_z:=g_{\vert X_z}:X_z\to Y_z$ is the MRC of $X_z$, and so $K_{Y_z}$ is pseudo-effective.
\end{corollary}

\begin{theorem}\label{algrc} Let $\cF\subset TX$ be a foliation such that $\mu_{\alpha,min}(\cF)>0$, for some $\alpha\in Mov(X)$. Then $\cF$ is algebraic, and its leaves have rationally connected closures.
\end{theorem}

\begin{proof} The algebraicity of $\cF$ follows from Theorem \ref{mu-alg}. Let $f=h\circ g$ be the relative MRC of $f$, a suitable fibration such that $\cF=Ker(df)^{sat}$. By Corollary \ref{KXz<0}.1, and Theorem \ref{tur}, $X_z$ is uniruled and so $dim(Y)<dim(X)$. By Corollary \ref{KXz<0}.2, $K_{Y_z}$ is not pseudo-effective if $dim(Y)>dim(Z)$, contradicting Corollary \ref{MRCrel}. Thus $Y=Z$, and the claim.
\end{proof}

\begin{corollary}\label{ur} The following are equivalent:

1. $X$ is uniruled.

2. $K_X$ is not pseudo-effective.

3. $h^0(X,mK_X+A)=0$ for any ample $A$ and large $m$.

4. $\mu_{\alpha,min}(\Omega^1_X)<0$, for some $\alpha \in Mov(X)$.
\end{corollary}

We have: 1 $\Longrightarrow$ 2,3 $\Longrightarrow$ 4 $\Longrightarrow$ 1 (the last two implications by Corollaries \ref{mumin>0} and \ref{algrc} respectively).

\begin{corollary}\label{rc} The following are equivalent:

1. $X$ is rationally connected

2. $\Omega^p_X$ is not pseudo-effective for any $p>0$.

3. $h^0(X,Sym^m(\Omega^1_X)\otimes A)=0$ for any $p>0$, $A$ ample and $m$ large.

4. $\mu_{\alpha,max}(\Omega^1_X)<0$ for some $\alpha\in Mov(X)$.
\end{corollary}

\begin{proof} The implications 1 $\Longrightarrow$ 2,3 are easy by taking a rational curve with ample normal bundle in $X$, 3 $\Longrightarrow$ 4 is seen by contradiction, from Theorem \ref{trc}, taking $L:=r^*(K_R)\subset \Omega^p_X$ if $p:=dim(R)>0$ is a pseudo-effective line bundle, thus such that $0<h^0(X, m.L+A)\leq h^0(X, Sym^m(\Omega^p_X)\otimes r^*(A))$ for infinitely many $m's$. That 4 $\Longrightarrow$ 1 follows from Theorem \ref{algrc} applied to $\cF=TX$.\end{proof}

\begin{remark}1. Theorem \ref{tur} (resp. Corollary \ref{rc}) claims that $X$ is uniruled (resp. rationally connected) if and only if there is a covering family of curves (possibly of large genus) on which $det(TX)$ (resp. $TX)$ is ample. 

2.  The `uniruledness conjecture' claims that $X$ is uniruled if $K_X$ is not $\Bbb Q$-effective. It implies that in the above two corollaries, the ample $A$ could be removed in the assertions 3, and that `pseudo-effective' could be replaced by `$\Bbb Q$-effective' in assertions 2. 
\end{remark}

\begin{corollary}\label{ur'}(\cite{CPa},\cite{CPe}) Let $s$ be a nonzero section of $\otimes^mTX\otimes L$, for some $m>0$ and $L$ a line bundle with $c_1(L)=0$. Assume that $s$ vanishes somewhere on $X$. Then $X$ is uniruled (i.e: covered by rational curves).
\end{corollary}

\begin{proof} It is sufficient to show that $\mu_{\alpha,max}(TX)>0$, or equivalently, that $\mu_{\alpha,max}(\otimes^mTX\otimes L)>0$ for some $\alpha\in Mov(X)$. Let $L$ be the subsheaf of rank one of $\otimes^mTX\otimes L$ generated by $s$, and let $\alpha$ be the geometrically movable class defined by a family of ample curves going through some point where $s$ vanishes. Thus $0<\mu_{\alpha}(L)\leq \mu_{\alpha,max}(\otimes^mTX\otimes L)$ as claimed.
\end{proof}

\begin{corollary}\label{amplenorm} Let $C\subset X$ be an irreducible projective curve such that $TX_{\vert C}$ is ample. Then
$X$ is rationally connected. 

\end{corollary}

This was observed in \cite{Gou} by a different approach. 

\begin{proof} Apply Theorem \ref{algrc} , observing that $[C]\in Mov(X)$, by considering the graph $C"$ of the embedding $j:C'\subset X$ in $C'\times X,C'$ the normalisation of $C$, which has normal bundle $TX_{\vert C'}$, so that Corollary \ref{rc} applies, and $[C"]\in Mov(C'\times X)$, so that $[C]=p_*([C"])\in Mov(X)$.

\end{proof}


\section{Birational stability of the cotangent bundle}\label{birst}

\begin{theorem}\label{qpseff} (\cite{CPa}) Assume that $K_X$ is pseudo-effective. Let $Q$ be a torsionfree quotient of $\otimes^m\Omega^1_X$, for some $m>0$. Then $det(Q)$ is pseudo-effective.
\end{theorem}

\begin{proof} $K_X$ being pseudo-effective, $\mu_{\alpha,min}(\Omega^1_X)\geq 0,\forall \alpha\in Mov(X)$, by Corollary \ref{ur}. By Theorem \ref{mu}, $\mu_{\alpha,min}(\otimes^m\Omega^1_X)=m.\mu_{\alpha,min}(\Omega^1_X)\geq 0$, $\forall \alpha\in Mov(X)$, which means that any torsionfree quotient of $\otimes^m(\Omega^1_X)$ is pseudo-effective.
\end{proof}

\begin{theorem}\label{Lpseff}(\cite{CPa}) Let $L$ a pseudo-effective line bundle on $X$. If there exists a nonzero sheaf morphism $\lambda: L\to \otimes^m\Omega^1_X\otimes K_X^{\otimes p}$ for some $m\geq 0,p>0$, $K_X$ is pseudo-effective. 
\end{theorem}

\begin{proof} We proceed by induction on $n=dim(X)$, the case $n=1$ being obvious. Assume $K_X$ is not pseudo-effective, there then exists  $\alpha\in Mov(X)$ such that $K_X.\alpha<0$, and $\cF\subset TX$ a foliation such that $\mu_{\alpha,min}(\cF)>0$. Thus $\cF$ is algebraic, of the form $Ker(df)^{sat}$ for some fibration $f:X\to Z$ with rationally connected fibres, by Theorem \ref{algrc}. 

If $dim(Z)=0$, we get $\cF=TX$, and so: $\mu_{\alpha,min}(TX)>0$, which implies that $0\leq L.\alpha\leq \mu_{\alpha,max}((\otimes^m\Omega^1_X)\otimes K_X^{\otimes p})=-m.\mu_{\alpha,min}(TX)+p.K_X.\alpha<0$. Contradiction. 

If $dim(Z)>0$, then $dim(X_z)<n$ and we may apply induction. Let $\lambda_z:=\lambda_{\vert X_z}: L_z\to (\otimes^m\Omega^1_X\otimes K_X^{\otimes p})_{\vert X_z}$ be the restriction. The exact sequence on $X_z$: $0\to T\to \Omega^1_{X\vert X_z}\to \Omega^1_{X_z}\to 0$, with $T$ a trivial bundle of rank $dim(Z)$, induces a filtration on $(\otimes^m\Omega^1_X\otimes K_X^{\otimes p})_{\vert X_z}$ with graded pieces isomorphic to $T^{\otimes (m-j)} \otimes^{j}\Omega^1_{X_z}\otimes K_{X_z}^{\otimes p}$, and we thus get for some $j\in \{0,\dots,m\}$ a nonzero sheaf morphism $L_{X_z}\to \otimes^{j}\Omega^1_{X_z}\otimes K_{X_z}^{\otimes p}$. The induction then implies that $K_{X_z}$ is pseudo-effective, since $L_{X_z}$ is, like $L$ on $X$, pseudo-effective on $X_z$, for $z$ general in $Z$ (i.e: ouside countably many strict Zariski-closed subsets). This contradicts the non-pseudoeffectivity of $K_X$.\end{proof}

\begin{corollary}\label{big}(\cite{CPa}) Let $L$ a big line bundle on $X$. If there exists a nonzero sheaf morphism $L\to \otimes^m\Omega^1_X$ for some $m>0$, $K_X$ is big. 
\end{corollary}

\begin{proof} We may assume that $L$ is saturated, let $Q:=(\otimes^m\Omega^1_X/L)$. If $K_X$ is pseudo-effective, $det(\otimes^m\Omega^1_X)=N.K_X=L+det(Q)$. Since $L$ is big and $det(Q)$ is pseudo-effective by Theorem \ref{qpseff}, $N.K_X$, and so $K_X$, is big.

We show that $K_X$ is pseudo-effective, which will imply the claim. Since $L$ is big, $k.L+K_X$ is effective for some large $k>0$, and $-K_X$ thus admits a nonzero morphism in $kL=(kL+K_X)-K_X$, and so by composition with $\lambda^{\otimes k}$, $-K_X$ maps nontrivially into $\otimes^{mk}\Omega^1_X$. Equivalently, $\mathcal{O}_X$ injects into $\otimes^{mk}\Omega^1_X\otimes K_X$. Theorem \ref{Lpseff} implies that $K_X$ is pseudo-effective.
\end{proof}

\begin{remark}\label{rkod} Let $X:=\Bbb P^k\times Z$, where $k>0$, and $Z$ is of general type and dimension $(n-k)\geq 0$. Then $K_X$ is not pseudo-effective, although $\Omega^{n-k}_X\subset \Omega^{1 \otimes (n-k)}_X$ contains $L:=p^*(K_Z)$, a line bundle of Kodaira dimension $0\leq (n-k)\leq (n-1)$. The assumption that $L$ is big cannot be weakened to any smaller Kodaira dimension in order to imply the pseudoeffectivity of $K_X$. 
\end{remark}

The argument used in Theorem \ref{qpseff} however extends to arbitrary numerical dimensions when $K_X$ is assumed to be pseudo-effective.

\begin{definition} Let $L, P, A$ be  line bundles on $X$, $A$ ample. 

1. $\kappa(X,L):=max\{k\in \Bbb Z\vert \overline{lim}_{m>0}(\frac{h^0(X,m.L)}{m^k})>0\}\in \{-\infty,0,\dots,n\}$

2. $\nu(X,L):=max\{A,k\in \Bbb Z\vert \overline{lim}_{m>0}(\frac{h^0(X,m.L+A)}{m^k})>0\}\in \{-\infty,0,\dots,n\}$
\end{definition}

The elementary properties of these two invariants are: 

\begin{proposition}\label{nuvsk} 0. $\nu(X,L)\geq \kappa(X,L)$

1. $L$ is pseudo-effective (resp. $\Bbb Q$-effective) iff $\nu(X,L)\geq 0$ (resp. $\kappa(X,L)\geq 0$.

2. $L$ is big (i.e: $\kappa(X,L)=n)$ iff $\nu(X,L)=n$.

3. If $P$ is pseudo-effective, $\nu(X,L+P)\geq \nu(X,L)$.

4. If $L$ is big and if $P$ is pseudo-effective, $L+P$ is big.
\end{proposition}

\begin{theorem}\label{birstab} Assume that $K_X$ is pseudo-effective. Let $L$ be a line bundle on $X$, and $L\subset \otimes^m\Omega^1_X$ a nonzero sheaf morphism. Then $\nu(X,L)\leq \nu(X,K_X)$.
\end{theorem}

\begin{proof} Let $Q:=(\otimes^m \Omega^1_X)/L^{sat}$. Since $det(Q)$ is pseudo-effective by Theorem \ref{qpseff}, $N.K_X=det(L^{sat})+det(Q)$ is the sum of $L^{sat}$ and $P:=det(Q)$ which is pseudo-effective. 

Thus $\nu(X,K_X)=\nu(X,N.K_X)=\nu(X,L^{sat}+P)\geq \nu(X,L)$.
\end{proof}

\begin{remark} 1. The analogue of Proposition \ref{nuvsk}.(3) is false in general for the Kodaira dimension The proof of Theorem \ref{birstab} thus does not apply to prove $\kappa(X,L)\leq \kappa(X,K_X)$ in the same situation. However:

2. A central conjecture of birational geometry claims that $\nu(X,K_X)=\kappa(X,K_X)$ for any $X$. It thus implies, together with Theorem \ref{birstab}, that $\kappa(X,K_X)\geq \nu(X,L)$ for any $m>0, L\subset \otimes^m\Omega^1_X)$ if $\nu(X,K_X)\geq 0$. 

3. The Remark  \ref{rkod} shows that the assumption that $K_X$ is pseudo-effective cannot be removed or weakened. 
\end{remark}

When $K_X$ is not pseudo-effective (i.e. when $X$ is uniruled), we have the general version:

\begin{corollary}\label{birstab'} Let $f:X\to Z$ be the MRC of $X$. Let $L\subset \otimes^m\Omega^1_X$ be a subsheaf of rank $1$. Then $\nu(X,L)\leq \nu(Z,K_Z)$.
\end{corollary}

\begin{proof} We can and shall assume that $L$ is saturated in $\otimes^m\Omega^1_X:=E$. Let $X_z:=f^{-1}(z),z\in Z$ is a generic smooth fibre of $f$, and let $A$ be an ample line bundle on $X$, with $L_{X_z}:=L_{\vert X_z}$.

\begin{lemma}\label{l} $\nu(X,L)=-\infty$, unless $L_{X_z}\subset f^*(\otimes^m\Omega^1_{Z,z})$ over $X_z$. 
\end{lemma}

\begin{proof}The exact sequence $0\to f^*(\Omega^1_{Z,z})\to \Omega^1_{X\vert X_z}\to \Omega^1_{X_z}\to 0$ induces on $\otimes^{m}\Omega^1_{X\vert X_z}$ a filtration with graded pieces isomorphic, for $0\leq j\leq m$, to $G_j:=\otimes^{m-j} f^*(\Omega^1_{Z,z})\otimes(\otimes^j\Omega^1_{X_z})$. Since $X_z$ is rationally connected, we may choose $\alpha\in Mov(X)$ such that $f_*(\alpha)=0$, and $\alpha$ restricts to $X_z$ such that $\mu_{\alpha, max}(\otimes^{j}\Omega^1_{X_z})=j.\mu_{\alpha, max}(\Omega^1_{X_z})<0$, for any $j>0$. If $L$ is not contained in $\otimes^m(f^*(\Omega^1_{Z,z}))$, there exists $j>0$ and an injective sheaf map $L\to G_j$. We thus get: $\mu_{\alpha}(L)=L. \alpha \leq \mu_{\alpha,max}(G_j)=j.\mu_{\alpha,max}(\Omega^1_{X_z})<0$, since $\mu_{\alpha,max}(\otimes^{m-j}( f^*(\Omega^1_{Z,z})))=0$, the bundle $f^*(\Omega^1_{Z,z})$ being trivial. Thus $L.\alpha<0$, and $\nu(X,L)=-\infty$.
\end{proof}

From lemma \ref{l}, we deduce that $L\subset f^*(\otimes^m\Omega^1_Z)^{sat}\subset \otimes^m\Omega^1_X$. We can now conclude from the following Lemma \ref{l'}, analogous to Lemma \ref{qpseff}, using the very same arguments used to deduce Theorem \ref{birstab} from Lemma \ref{qpseff}.

\begin{lemma}\label{l'} Let $f:X\to Z$ be the MRC fibration of $X$, and let $\cG\subsetneq f^*(\otimes^m \Omega^1_Z)^{sat}$ be a saturated subsheaf. Then $det(\cQ)$ is pseudo-effective on $X$, where $\cQ$ is the quotient of $(f^*(\otimes^m \Omega^1_Z))^{sat}$ by $\cG$.
\end{lemma}

\begin{proof} We may and shall assume $Z$ to be smooth. By contradiction, if the conclusion does not hold, $\mu_{\alpha,min}(f^*(\otimes^m\Omega^1_Z)^{sat}))<0$ for some $\alpha\in Mov(X)$, and $-\mu_{\alpha,max}((f^*(T_Z))^{sat})=\mu_{\alpha,min}((f^*(\Omega^1_Z)^{sat}))<0$. There is then $\cF\subset (f^*(T_Z))^{sat}$, $\alpha$-semi-stable with $\mu_{\alpha,min}(\cF)>0$ which is an algebraic foliation on $X$ with closures of leaves rationally connected, and mapped non-trivially on $Z$ by $f$, contradicting the non-uniruledness of $Z$.
\end{proof} \end{proof}


\section{Shafarevich-Viehweg conjecture.}\label{SVconj}

Let $f:Y\to B$ be a projective holomorphic submersion with connected fibres between connected complex quasi-projective manifolds $Y,B$.

We assume that all fibres of $f$ have an ample canonical bundle, and that the `variation' of $f$, defined by: $Var(f)=rk(ks:TB\to R^1f_*(T_{Y/B}))$ is maximal, equal to $dim(B)$.

The situation considered by I. Shafarevich in 1962, was when $f$ was a non-isotrivial family of curves of genus at least $2$ on a curve $B$. His conjecture (proved by Parshin and Arakelov) was (formulated differently) that  $K_{\bar{B}}+D$ was big on $\bar{B}=B\cup D$, a compactification of $B$.

The conjecture of Shafarevich was extended by Viehweg in the following form: in the above situation, $K_{\bar{B}}+D$ is big if $\bar{B}$ is a smooth projective compactification of $B$ obtained by adding to $B$ a divisor $D$ of simple normal crossings.

In this situation, Viehweg-Zuo (\cite{VZ}) proved the existence of a big line bundle (the `Viehweg-Zuo sheaf') $L\subset \otimes^m \Omega^1_{\bar{B}}(Log D)$ which, combined with the following Theorem, implies Viehweg's conjecture (proved first in \cite{Taji}):

\begin{theorem} Let $X$ be a connected complex projective and $D$ a divisor of simple normal crossings on $X$. Assume that a big line bundle $L$ on $X$ has an injective sheaf morphism $L\subset \otimes^m \Omega^1_{X}(Log D)$. Then $K_X+D$ is big.
\end{theorem}

This is the logarithmic analogue of Theorem \ref{big}. It is proved by extending to the logarithmic setting the theorems \ref{algcrit} and \ref{big}, which is done very straightforwardly, by replacing in the proofs the (co)tangent bundles by their $Log$-analogs.

The Viehweg's conjecture has been extended, and proofs given, in various directions (fibres of general type, or with semi-ample canonical bundle, or $B$ `special', and then $f$ is isotrivial). In another direction, inspired by Lang's conjectures, the Brody, Kobayashi, or Picard hyperbolicity of the corresponding moduli spaces have been proved in various situations. See \cite{D} for references on this topic.


\section{Numerically trivial foliations.}

Let $\cF\subset TX$ be a numerically trivial and pseudo-effective foliation, that is: such that $det(\cF).A^{n-1}=0$, for some ample line bundle $A$. Since $c_1(\cF)=0$, the pseudo-effectivity of $\cF$ and $\cF^*$ are equivalent. 

\medskip

When $X$ is an abelian variety and $\cF$ is linear (that is: $\cF=Ker(w),$ $w\in H^0(X,\Omega^{n-r}))$, $\cF$ is numerically flat and pseudo-effective,
but the leaves of $\cF$ are in general not algebraic. In such cases however, $\cF$ is flat (i.e: given by a linear representation of the fundamental group), and so semi-stable but not stable.

\medskip

The following question/conjecture was raised by Pereira-Touzet (who gave several affirmative answers, mostly in the corank-one case): if $\cF$ is $A$-stable and numerically trivial, is it algebraic?

\medskip

Theorem \ref{algcrit}, reduces this question to the non-pseudoeffectivity of $\cF^*$ when $c_1(\cF)=0$, and $\cF$ is $A$-stable.

A partial solution is given in the following:

\begin{theorem}\label{ccp} (\cite{HP},\cite{ccp}) Let $E$ be a vector bundle on $X$, assume that $c_1(E)=0\neq c_2(E)$, and that $Sym^m(E)$ is $A$-stable for some ample $A$ and some $m> r=rk(E)$. Then $E$ is not pseudo-effective\footnote{The proof works for $X$ compact K\"ahler, and $E$ a reflexive sheaf on $X$.}.
\end{theorem}

The proof of \cite{ccp} relies on analytic and metric methods. The previous proof given in \cite{HP} by H\"oring-Peternell is mostly algebro-geometric. It runs along the following lines, starting with $E$ pseudo-effective, with $c_1(E)=0$, and all $Sym^m(E)$ $A$-stable, in order to show that $c_2(E)=0$. The first step shows, by induction on the codimension of $W\subset P:=\Bbb P(E)$, a component of the restricted base locus $B_{-}(L), L:=\mathcal{O}_P(1)$, that $p(W)$ has codimension at least $2$ in $X$, $p:P\to X$ being the natural projection. The key ingredient in this step is following lemma, due to Mumford when $r=rk(E)=2$:

\begin{lemma}(\cite{HP}) Let $E$ be a vector bundle on a smooth projective curve $C$. Let $L:=\mathcal{O}_{\Bbb P(E)}(1)$. Assume that $c_1(E)=0$, and $Sym^m(E)$ is stable for all $m>0$. Then $L^d.Z>0$ for all $d>0$, and any irreducible $Z\subset \Bbb P(E)$ of codimension $d$. 
\end{lemma}

The second step then deduces from the first one that the restriction $E_S$ of $E$ to a general complete intersection surface $S\subset X$  of large degree, is nef. So that $E$ and $det(E^*)=det(E)=0$ are nef, and $E^*$ is thus nef too. Hence $E$ is numerically flat, so in particular all Chern classes of $E$ vanish.

From Theorem \ref{ccp} and Theorem \ref{algcrit}, we get:

\begin{corollary}\label{algc1=0} Let $\cF\subset TX$ be a foliation with $c_1(\cF)=0\neq c_2(\cF)$. If $Sym^{[m]}(\cF)$ is $A$-stable for some $m>rk(\cF)$, $\cF$ is algebraic.
\end{corollary}

The Corollary \ref{algc1=0} (or the slightly weaker version of \cite{HP}) now permits to give an alternative proof of the Beauville-Bogomolov-Yau\footnote{Usually called the Bogomolov-Beauville decomposition. Since Yau's Ricci-flat metric is essential here, BBY seems justified.} (BBY) decomposition theorem in the projective and klt singular case.


\section{The Beauville-Bogomolov-Yau decomposition}

Let $X$ be an $n$-dimensional compact K\"ahler manifold with $c_1(X)=0$. By Yau's theorem, any such $X$ carries a Ricci-flat K\"ahler metric. There are $3$ basic examples of such manifolds, each class characterised by the restricted holonomy group $Hol^0$ of any of their K\"ahler, Ricci-flat, metrics: $Hol^0$ is trivial for 1, is $Sp(4r)$ for 2, and is $SU(n)$ for 3.

1. The compact complex tori are the quotients  $\C^n/\Gamma$, where $\Gamma$ is a cocompact lattice of $\C^n$. 

2. The irreducible hyperk\"ahler (HK) manifolds are defined as the ones which are simply-connected, even-dimensional: $n=2r$, and admit 
a holomorphic symplectic $2$-form $\sigma$ generating $H^0(X,\Omega^2_X)$, and such that $s^{\wedge r}$ is a nowhere vanishing section of $K_X$. There are few known examples: $2$ deformation classes in each even dimension due to Beauville (and Fujiki in dimension $4$), and one additional in dimensions 6 and 10, found by O'Grady. 

3. The Calabi-Yau (CY) manifolds are those which are simply connected, with $h^{p,0}=0,\forall p=1,\dots,(n-1)$, and with $K_X$ trivial. Many deformation families known in each dimension. Among these complete intersections of suitable multidegrees in the projective spaces.

\

The BBY decomposition is the following:

\begin{theorem}(\cite{beauv}) Let $X$ be a connected compact K\"ahler manifold with $c_1(X)=0$. A suitable finite \'etale cover of $X$ is a 
product of a complex compact torus $T$, and of simply-connected CY or HK manifolds.
\end{theorem}

Let us sketch the classical proof given in \cite{beauv}: We equip $X$ with some Ricci-flat K\"ahler metric (\cite{YAU}). Let $Hol^0$ be its restricted holonomy representation, and $T_X=F\oplus (\oplus_i T_i)$ be the splitting of the tangent bundle of $X$ into factors which are irreducible for the action of $Hol^0$. This splitting is well-defined locally on $X$, and globally only up to possible reordering. These local factors correspond also to a local splitting of $X$ into a direct product of K\"ahler submanifolds. By Berger general classification of the holonomy representations , only a flat factor $F$ appears, together with other factors $T_i$ having restricted holonomy either $Sp$ or $SU$ (the ones susceptible to preserve a section of $K_X$ by parallel transport).

We lift this decomposition to the universal cover $\widetilde{X}$ of $X$ on which the metric is complete. By De Rham decomposition theorem, $\widetilde{X}$ splits as a K\"ahler product, corresponding to these $Hol^0$-factors (\cite[Theorem 10.43]{bes}). The Cheeger-Gromoll theorem then says that $\widetilde{X}=\C^k\times P$, where $\C^k$ corresponds the flat factor, the other factor $P$ being compact, and corresponding to the product of the factors with $Hol^0$ either $Sp$ or $SU$ (\cite[Theorem 6.65]{bes}). Bieberbach's theorem (\cite{bieber}) then  concludes the proof.

Notice that this shows in particular that the (everywhere regular) foliations defined on $X$ by the factors $T_i$ have compact leaves (with finite fundamental groups), and are thus `algebraic' (i.e: with equal Zariski and topological closures of their leaves).

We shall now give, but only in the projective case, an alternative proof which does not require the consideration of the universal cover. This is done by proving directly, using Theorem \ref{ccp}, that the foliations defined by the $T_i's$ are all algebraic. This follows from the following:

\begin{lemma} For each $i$, and each $m>0,$ $c_1(T_i)=0$, and $Sym^m(T_i)$ is $A$-stable for any ample line bundle $A$ on $X$. 
\end{lemma} 

\begin{proof} Since $c_1(TX)=c_1(F)=0$, the first claim follows from the second one. The second claim is a consequence of representation theory of the groups $SU$ and $Sp$ (\cite{fh}).
\end{proof}

If we take thus the (regular and submersive) fibration $f_i:X\to B$ with fibres the (projective) leaves of $T_i$, it is isotrivial of fibre $F_i$ (by the existence of the transversal foliation defined by $F\oplus_{j\neq i} T_j)$, and $F_i$ has a vanishing irregularity and a discrete automorphism group (because the restricted holonomy is either $SU$ or $Sp)$, which implies that after a finite \'etale base change $B'\to B$, $X$ splits as a product $F_i\times B'$. 

Induction on $dim(X)$ now reduces the proof to the case where $TX=F$, the flat factor. Bieberbach theorem (\cite{bieber}) then permits to conclude. For details, see \cite{Ca20}.

This BBY decomposition is still valid for projective varieties with $klt$ singularities, but the proof, obtained first in \cite{HP}, and using a combination of the works \cite{EGZ}, \cite {GKP}, \cite{GGK}, \cite{Dr}, is much more involved, both technically and conceptually.

A simplified proof, which closely follows the above alternative proof in the smooth case, and avoids the delicate arguments in positive characteristic used in \cite{Dr}, can be found in \cite{Ca20}.
The result has then been extended to the K\"ahler case in \cite{BGL}, by deformation to the projective case.



\section{Some questions}

We denote by $X$ a smooth and connected complex projective manifold\footnote{The questions raised below could be extended to the compact K\"ahler case, and to smooth or klt orbifold pairs.}, of dimension $n$. Let $\cE$ be a torsionfree coherent sheaf on $X$.



\medskip

\subsection{Pseudoeffectivity of determinants.}

\begin{definition} We say that $\cE$ is pseudo-effective if so is $det(Q)$, for each nonzero quotient of $\cE^{**}$. Equivalently, $\mu_{\alpha,min}(\cE^{**})\geq 0, \forall \alpha\in Mov(X)$.
\end{definition} 

\begin{example} 1. If $\cE$ is pseudo-effective, so is any of its quotients, as well as $\otimes^m\cE, Sym^m\cE, \wedge^m\cE, \forall m>0$, and more generally the tensors deduced from $\cE$ by linear representations, by Theorem \ref{mu}.

2. If $K_X=det(\Omega^1_X)$ is pseudo-effective, so is $\Omega^1_X$, as seen in Theorem \ref{qpseff}. In \cite{CPa}, the same property is shown, more generally, for any quotient of $\otimes^m(\Omega^1_X)$ with pseudo-effective determinant.
\end{example}

We ask the same question also for subsheaves of $\otimes^m(\Omega^1_X)$:

\begin{question} Let $\cE\subset \otimes^m\Omega^1_X$ be such that $det(\cE)$ is pseudo-effective. Is $\cE$ pseudo-effective? This is a birational version of stability.
\end{question}

Notice that if $\cE\subset \Omega^p_X$ is saturated and defines an {\it algebraic} foliation of rank $r=(n-p)$, the answer is yes by \cite{CPa}, since $\cE$ is then the cotangent sheaf of the space of leaves.

This question might be answered by showing the relative algebraicity of two transcendental foliations deduced from suitable quotients of $\Omega^1_X$, using a suitable extension of the next algebraicity criterion in \cite{CPa}.

\medskip

Recall this algebraicity criterion (theorem \ref{Mu-alg}): if $\cF\subset TX$ is a foliation such that $\cF^*$ (which is a quotient of $\Omega^1_X$) is not pseudo-effective, the leaves of $\cF$ are algebraic. 
Moreover, if $\mu_{\alpha,max}(\cF^*)<0$ for some $\alpha\in Mov(X)$, the leaves of $\cF$ are rationally connected.

\begin{question} If $\cF^*$ is not pseudo-effective, are there restrictions on the leaves of $\cF$, and which ones?
\end{question}

Already when $\cF=TX$, the answer is nontrivial, and consists in the classification of manifolds $X$ with $\Omega^1_X$ not pseudo-effective, or equivalently with $\nu_1(X)=-\infty$, using the notations of the next subsection.
Although such manifolds may have an ample, or a trivial canonical bundle, as shown by hypersurfaces in $\Bbb P^{n+1}$, they share some common properties with the rationally connected manifolds. For example, rationally connected manifolds are simply-connected, and it is shown in \cite{BKT} that the linear representations of $X$ have a finite image if $\Omega^1_X$ is not pseudo-effective. For manifolds of general type, even for surfaces, the invariant $\kappa_1$ (see below), does not seem to be described by, or even related to, other algebro-geometric invariants. For special manifolds, in the sense of definition \ref{dspec} below, the situation might be much simpler (see conjecture \ref{qspec}).

An additional problem is that if the foliation $\cF$ has $\cF^*$ not pseudo-effective, and so if $\cF$ is algebraic, it is not clear whether or not the general fibres of an associated fibration have $\Omega^1$ not pseudo-effective. The proof when $\mu_{\alpha,min}(\cF)>0$ relies on the MRC, and it is unknown whether a similar fibration may be expected to exist on any $X$, with fibres having $\Omega^1$ not pseudo-effective, and base with $\Omega^1$ pseudo-effective.



\medskip

\subsection{Variations on Abundance.}

Let $\cE$ be torsionfree coherent on $X$, of rank $r$, and let $\Bbb P(\cE), \mathcal{O}_{\cE}(1)$ be the corresponding (Grothendieck) projectivisation, together with its tautological line bundle. 

\begin{definition} We write $\kappa(\cE):=\kappa(\Bbb P(\cE), \mathcal{O}_{\cE}(1))-(r-1)$, and $\nu(\cE):=\nu(\Bbb P(\cE), \mathcal{O}_{\cE}(1))-(r-1)$. We also write $\nu_p(X):=\nu(\Omega^p_X)$, and $\kappa_p(X):=\kappa(\Omega^p_X)$, so that: $\nu(X)=\nu_n(X)$, and $\kappa(X):=\kappa_n(X)$. 
\end{definition}

The definition fits with the original one for line bundles. Moreover, the invariants $\kappa_p$ and $\nu_p$ are preserved under finite \'etale covers. The invariant $\kappa_1$ has been introduced and studied by F. Sakai in \cite{Sak}\footnote{F. Sakai uses the `invariant $\lambda$', which takes the value $-n$ when $\kappa_1=-\infty$, but otherwise coincides with $\kappa_1$. His convention simplifies the formulations in some situations, such as products.} . 

The classical `Abundance conjecture' $\nu_n(X)=\kappa_n(X)$ extends as:

\begin{conjecture}\label{abdc} For any $X$ and $p>0$, one has: $\nu_p(X)=\kappa_p(X)$. 
\end{conjecture}

\begin{remark} 1. In \cite{HP20}, this is stated when $\kappa_p(X)=-\infty$, and proved for elliptic surfaces (except for some isotrivial ones), and some other additional cases.

2. If $\kappa_p(X)=-\infty, \forall p>0$, and if we assume the classical `Abundance conjecture', we get by the MRC fibration that $X$ is rationally connected, and so that $-\infty=\kappa_p(X)=\nu_p(X), \forall p>0$. Conjecture \ref{abdc} is thus a consequence of its `classical' version in this situation. 
\end{remark}

\begin{example} 1. If $c_1(X)=0$, from the Beauville-Bogomolov-Yau decomposition theorem, and \cite{HP}, one can easily see\footnote{In the next formula,  if $\tilde{q}(X)=0$, replace Sakai's $-n$ by $-\infty$.} that $\nu_1(X)=\kappa_1(X)=\tilde{q}(X)-n$, where $\tilde{q}(X)$ is the maximum of $q(X')$ when $X'$ runs through the finite \'etale covers of $X$. The same properties can be still shown with more work\footnote{Because one needs to carefully distinguish \'etale from quasi-\'etale covers, as shown by the Kummer $K3's$.}  when $X$ is a smooth model of a klt variety with $c_1=0$. Thus $\kappa_1(X)=\nu_1(X)=\tilde{q}(X)-n$ whenever $\kappa(X)=0$, if one assumes the existence of good minimal models.
\end{example}

\begin{remark}\label{brQ}  The Abundance conjecture $\nu=\kappa$ fails in general for  rank-one subsheaves $L$ of $\Omega^p_X$, by the examples of M. Brunella and M. Mc Quillan of the (duals of the) tautological foliations on irreducible free quotients of the bidisc, which have $\kappa=-\infty$ and $\nu=1$. 
\end{remark}

\begin{question} Are there rank-one subsheaves $L$ of $\Omega^p_X$ for suitable $X, p>0$, with $p=\nu(X,L)>\kappa(X,L)\geq 0$? As seen just above, there are examples with $p=\nu(X,L)>\kappa(X,L)=-\infty$.
\end{question}



\medskip

\subsection{Special manifolds.}

Recall that from Bogomolov's theorem that $\kappa(X,L)\leq p, \forall X, \forall p>0$, if $L$ is any rank-one subsheaf of $\Omega^p_X$.

\begin{definition}\label{dspec} We say that $X$ is special\footnote{We may assume $X$ to be compact K\"ahler, here.} if $\kappa(X,L)<p$, for any such $L$, and any $p>0$; 
\end{definition}

Bogomolov's theorem has been strengthened by C. Mourougane and S. Boucksom to $\nu(X,L)\leq p, \forall X, \forall p>0$. 

\begin{conjecture} If $X$ is special, $\nu(X,L)<p, \forall L\subset \Omega^p_X,\forall p>0$, $L$ of rank one. 
\end{conjecture}

This has been proved in \cite{PRT21} for $p=1$. Notice that if $L\subset \Omega^1_X$ is of rank one with $\nu(X,L)=1$, there may be no $L'\subset  \Omega^1_X$ with $\kappa(X,L')=1$, as shown by the example of Brunella and Mc Quillan in remark \ref{brQ}, since no finite \'etale cover of the ambiant surface maps onto a curve of positive genus. 

\begin{conjecture}\label{qspec} Assume $X$ (smooth compact K\"ahler) is special. Do we have\footnote{If $\tilde{q}(X)=0$, replace the RHS $-n$ by $-\infty$.} : $\nu_1(X)=\kappa_1(X)=\tilde{q}(X)-n$? 

If yes, then $\nu_1(X)=\kappa_1(X)=-\infty$, that is $\Omega^1_X$ not pseudo-effective, if and only if $\tilde{q}(X)=0$, if and only if $\pi_1(X)$ is finite.  \end{conjecture}

Recall that the Abelianity conjecture claims that $\pi_1(X)$ is virtually abelian if $X$ is special, so that $\tilde{q}(X)=0$ should then mean that $\pi_1(X)$ is finite.

Notice that if $X$ is a compact K\"ahler complex manifold with all of its finite \'etale covers having a surjective Albanese map (as is the case when $X$ is special), then $\kappa_1(X)\geq \tilde{q}(X)-n$. 

\medskip

{\bf Acknowledgement:} We thank the reviewers for their careful readings, corrections and suggestions, which lead to improvements of the text.

\end{document}